\documentclass[a4,11pt]{amsart}

\usepackage{latexsym}
\usepackage{amsmath}
\usepackage{amssymb}
\usepackage{amsthm}
\usepackage[mathcal]{eucal}

\usepackage{enumerate}
\allowdisplaybreaks

 \makeatletter \@addtoreset{equation}{section}

\makeatother \makeatletter

\newtheorem{theorem}{Theorem}[section]

\newtheorem{lemma}[theorem]{Lemma}
\newtheorem{proposition}[theorem]{Proposition}
\theoremstyle{definition}

\newtheorem{remark}[theorem]{Remark}

\newcommand{\R}{{\mathbb R}}

\newcommand{\bd}{\begin{defi}}
\newcommand{\ed}{\end{defi}}
\newcommand{\nnm}{\nonumber}
\newcommand{\be}{\begin{equation}}
\newcommand{\ee}{\end{equation}}
\newcommand{\barr}{\begin{array}}
\newcommand{\earr}{\end{array}}
\newcommand{\bmn}{\begin{eqnarray}}
\newcommand{\emn}{\end{eqnarray}}
\newcommand{\bnm}{\begin{eqnarray*}}
\newcommand{\enm}{\end{eqnarray*}}
\newcommand{\bln}{\begin{subequations}}
\newcommand{\eln}{\end{subequations}}

\newcommand{\ba}{\begin{align}}
\newcommand{\ea}{\end{align}}
\newcommand{\banm}{\begin{align*}}
\newcommand{\eanm}{\end{align*}}

\newcommand{\f}{\frac}
\newcommand{\df}{\dfrac}

\begin{document}

\title[Heston Model]{Analytic approach to solve a degenerate parabolic PDE for the Heston model}
\author{A. Canale}
\address{Dipartimento di Matematica, Universit\`a degli Studi di Salerno, Via Giovanni Paolo II, 132, I 84084 FISCIANO (Sa), Italy.}
\email{acanale@unisa.it}
\author{R.M. Mininni}
\address{Dipartimento di Matematica Universit\`a degli Studi di Bari A. Moro, Via E. Orabona 4, 70125 Bari, Italy.}
\email{rosamaria.mininni@uniba.it}
\author{A. Rhandi}
\thanks{The third author has been supported by the M.I.U.R. research project
Prin 2010MXMAJR}
\address{Dipartimento di Ingegneria dell'Informazione, Ingegneria Elettrica e Matematica Applicata, Universit\`a di Salerno, Via Giovanni Paolo II, 132, I 84084 FISCIANO (Sa), Italy.}
\email{arhandi@unisa.it}


\maketitle

\begin{abstract}
We present an analytic approach to solve a degenerate parabolic problem associated to the Heston model, which is widely used in mathematical finance to derive the price of an European option on an risky asset with stochastic volatility. We give a variational formulation, involving weighted Sobolev spaces, of the second order degenerate elliptic operator of the parabolic PDE. We use this approach to prove, under appropriate assumptions on some involved unknown parameters, the existence and uniqueness of weak solutions to the parabolic problem on unbounded subdomains of the half-plane.
\end{abstract}

\maketitle

\section{Introduction}
Heston in \cite{H} derived the pricing formula of a stock European option when the price process $\{S_t, t\ge 0\}$ of the underlying asset satisfies the following stochastic differential equation (SDE)

\be\label{bseq}
dS_t = \eta\, S_t\, dt + \sqrt{Y_t}\, S_t\, dW_t, \quad t\ge 0,
\ee
where the constant parameter $\eta\in\R$ denotes the instantaneous mean return of the underlying asset, and, contrary to the original Black and Scholes model for European options \cite{BS}, the non-constant volatility $\sqrt{Y_t}$ is supposed to be stochastic. The variance process $Y=\{Y_t, t\ge 0\}$ is assumed to be a diffusion process whose dynamics is described by the following SDE
\be\label{cireq}
dY_t = \kappa(m - Y_t)\, dt + \sigma\, \sqrt{Y_t}\, dZ_t, \quad t\ge 0,
\ee
used in mathematical finance by Cox et al. \cite{CIR} to model ``short-term interest rates" of zero-coupon bonds. The parameters $\kappa, m$ and $\sigma$ are supposed to be positive constants. The process $Y$ is known in literature as \emph{CIR process} or \emph{square-root process}. In particular, $m$ is the long-run mean value of $Y_t$ as $t\to\infty$, $\kappa$ is called the ``rate of mean reversion" that is, $\kappa$ determines the speed of adjustment of the sample paths of $Y$ toward $m$, and $\sigma$ is the constant volatility of variance (often called the  \emph{volatility of volatility}). The state space of the diffusion $Y$ is the interval $[0,\infty)$.

The processes $\{W_t, t\ge 0\}$ and $\{Z_t, t\ge 0\}$ in \eqref{bseq} and \eqref{cireq} are standard one-dimensional Brownian motions. They are supposed to be correlated
$$
dW_t\, dZ_t = \rho\, dt,
$$
where $\rho\in [-1,1]$ denotes the instantaneous correlation coefficient.

Using the two-dimensional Ito's formula (cf., for example, \cite[Chap. IV.32]{RW}), the price $U$ of an European option with a risky underlying asset, fixed maturity date $T>0$ and exercise price $K>0$ satisfies the following degenerate parabolic problem
\smallskip
\be\label{eq1}
\small
\left\{\begin{array}{ll}
\df{\partial U}{\partial t} + \df{1}{2} y S^2\df{\partial^2 U}{\partial S^2} + \df{1}{2}y \sigma^2\df{\partial^2 U}{\partial y^2} + \rho \sigma y S \df{\partial^2 U}{\partial S\partial y} + \kappa(m-y)\df{\partial U}{\partial y}& + r(S\df{\partial U}{\partial S}-U)= 0,\\[2.5\jot]
 & \hbox{in $[0,T)\times [0,\infty)^2$}\\[3.5\jot]
U(T,S,y) = h(S) & \hbox{in $[0,\infty)^2$},
\end{array}\right.
\ee

\medskip
\noindent with the final pay-off of the option as the terminal condition, namely
$$
h(S)=(S - K)_+ \quad\text{or}\quad h(S)=(K - S)_+
$$
corresponding to European call and put options, respectively. The price $U:=U(t,S,y)$ depends on time $t$, on the stock price variable $S$ and on the variance variable $y$.

The degenerate parabolic problem \eqref{eq1} is obtained imposing some assumptions about the financial market, as the no-arbitrage condition i.e., given the evolutions of $S_t$ and of $Y_t$, the European option is priced in such a way that there are no opportunities to make money from nothing.

The PDE in \eqref{eq1} has degenerate coefficients in the $S$ variable and possibly also in the $y$ variable. In order to remove the degeneracy with respect to the variable $S$, we define the stochastic process $\{X_t, t\ge 0\}$ as follows
$$
X_t = \ln{\left(\df{S_t}{S_0}\right)}, \quad t\ge 0.
$$

\noindent
Further, consider the following function
$$
\widetilde{u}(t,S,y):= U(t,S,y)-e^{-r(T-t)}h(Se^{r(T-t)}),
$$
 which indicates the excess to discounted pay-off. The parameter $r\ge 0$ denotes the constant risk-neutral interest rate, As observed by Hilber et al. in \cite{HMS}, according to the boundary conditions on the PDE in \eqref{eq1} suggested in \cite{H}, $\widetilde{u}$ decays to zero as $S\to 0$ and $S\to\infty$. Then, by changing the time $t\to T-t$, setting $x=\ln{S}$ (assume $S_0 = 1$), and using the following transformation
\begin{equation}\label{transf}
u(t,x,y):=e^{-\f{\omega}{2}y^2}\, \widetilde{u}(T-t,S,y),\quad \omega>0,
\end{equation}
we deduce from (\ref{eq1}) that the function $u$ satisfies the following initial value forward parabolic problem
\begin{equation}\label{eq2}
\left\{\begin{array}{ll}
\df{\partial u}{\partial t}(t,x,y) = - (\mathcal{L}^Hu)(t,x,y) + F(t,y), & \hbox{$t\in (0,T],\, (x,y)\in\Omega$} \\[3.5\jot]
u(0,x,y)=0, & \hbox{$(x,y)\in\Omega$},
\end{array}
\right.
\end{equation}
where $\Omega =\mathbb{R}\times [0,\infty)$. The operator $\mathcal{L}^H$ is given by
\begin{align}\label{lhop}
(\mathcal{L}^H\varphi)(x,y) &= -\df{1}{2}y\df{\partial^2 \varphi}{\partial x^2}-\df{1}{2}\sigma^2 y\df{\partial^2 \varphi}{\partial y^2}-\rho \sigma y\df{\partial^2 \varphi}{\partial x\partial y}\nnm\\[1.5\jot]
&- (\omega \rho \sigma y^2-\df{1}{2}y+r)\df{\partial \varphi}{\partial x}
- [\omega \sigma^2 y^2+\kappa(m-y)]\frac{\partial \varphi}{\partial y}\nnm\\[1.5\jot]
&-\left[\frac{1}{2}\omega \sigma^2 y(\omega y^2+1)+\omega y\kappa(m-y)-r\right]\varphi
\end{align}
and
$$
F(t,y)=\df{K}{2}ye^{-rt}e^{-\f{\omega}{2} y^2}\delta_{\ln{K} - rt}.
$$

The motivation to consider the transformation \eqref{transf} is explained in \cite{HMS}, taking into account that the price $U$ remains bounded for all $y$ (cf. \cite{H}).

To our knowledge, the use of a variational approach to prove existence and uniqueness of solutions to these pricing problems is very recent. Achdou et al. \cite{AT}-\cite{AFT} used variational analysis using appropriate weighted Sobolev spaces to solve parabolic problems connected to option pricing when the variance process $Y$ is a function of a mean reverting Ornstein-Uhlenbech (OU) process. Successively, proceedings as in the previous works, Hilber et al. \cite{HMS} used variational formulation to present numerical solutions by a sparse wavelet finite element method to pricing problems in terms of parabolic PDEs when the volatility is modeled by a OU process or a CIR process. Daskalopoulos and Feehan \cite{DF} used variational analysis with the aid of weighted Sobolev spaces to prove the existence, uniqueness and global regularity of solutions to obstacle problems for the Heston model, which in mathematical finance correspond to solve pricing problems for perpetual American options on underlying risky assets.

Observe that by applying a space-time transformation, the diffusion $Y$ follows the dynamics of a squared Bessel process with dimension
$$
\alpha=\df{4\kappa m}{\sigma^2} > 0
$$
(cf. \cite[Section 6.3]{JYC}). It is known (cf. \cite[Chap. V.48]{RW}) that for $\alpha>2$ a general $\alpha$-dimensional squared Bessel process starting from a positive initial point stays strictly positive and tends to infinity almost surely as time approaches infinity while, for $\alpha=2$ the process is strictly positive but gets arbitrarily close to zero and $\infty$, and for $\alpha > 2$ the process hits zero 0 recurrently but will not stay at zero, i.e. the 0-boundary is strongly reflecting.

Ideally, one would like a variance process which is strictly positive, because otherwise it degenerates to a deterministic function for the time it stays at zero. Then to translate this property to the CIR process $Y$, we assume the condition
\be\label{feller-condition}
\kappa m>\f{\sigma^2}{2}.
\ee
Simulation studies to investigate numerically how the effect of varying the correlation $\rho$ (cf. \cite{G}) and the volatility parameter $\sigma$ (cf. \cite{M}) impacts on the shape of the implied volatility curve in the Heston model, clearly show that under the condition \eqref{feller-condition} the volatility $\sqrt{Y_t}$ always remains strictly positive.


Thus, the above arguments let us to assume $y\in [a,\infty)$ with a sufficiently small $a>0$, in order to remove the degeneracy at zero with respect to the variable $y$ and take $\Omega=\mathbb{R}\times [a,\infty)$ in \eqref{eq2}.

By using the variational formulation of the parabolic PDE in \eqref{eq2} performed in \cite{HMS}, the aim of the present paper is to use form methods to prove the existence and uniqueness of a weak solution to the problem \eqref{eq2} and to study the existence of a positive and analytic semigroup generated by $-\mathcal{L}^H$, with an appropriate domain, in a weighted $L^2$-space with suitable weights $\phi$ and $\psi$.

The article is organized as follows. In Section 2 we define the Hilbert and weighted Sobolev spaces
we shall need throughout this article, describe our assumptions on the Heston operator coefficients and prove
the continuity estimate for the sesquilinear form defined by the operator $\mathcal{L}^H$ given in \eqref{lhop}, with Dirichlet boundary conditions. In Section 3 we derive Garding's inequality for the sesquilinear form, and deduce the existence of a unique weak solution to the problem \eqref{eq2}. We obtain also that the realization of $-\mathcal{L}^H$ in $L^2$ with Dirichlet boundary conditions generates an analytic semigroup $(e^{-t\mathcal{L}^H})$. The positivity of $(e^{-t\mathcal{L}^H})$ can be proved applying the first Beurling-Deny criteria.

\section{Heston model: the variational formulation}
Throughout this article, the coefficients of the operator $\mathcal{L}^H$ are required to obey the Feller condition \eqref{feller-condition} and $\Omega=\mathbb{R}\times [a,\infty)$ with some positive constant $a$.


We propose to use form methods to solve the parabolic PDE in (\ref{eq2}). To this purpose we consider the weight functions
\[
\phi(x)=e^{\nu|x|}, \qquad \psi(y)=e^{\frac{\mu}{2}y^2},\quad (x,y)\in\Omega,
\qquad \nu,\,\mu>0,
\]
and define the Hilbert space
$$L^2_{\phi, \psi}(\Omega)=
\{v\,|\,(x,y)\mapsto v(x,y)\phi(x)\psi(y)\in L^2(\Omega)\}$$
equipped with the weighted $L^2$-norm
$$\|v\|_{\phi ,\psi}=\left(\int_\Omega |v(x,y)|^2\phi^2(x)\psi^2(y)\,dx\,dy\right)^{\frac{1}{2}}.$$
Furthermore we define the weighted Sobolev space
\[
V_{\phi,\psi}=
\Big\{v\,\Big|\quad \Big(v,\sqrt y{\frac{\partial v}{\partial x}},
\sqrt {y}{\frac{\partial v}{\partial y}}\Big)
\in (L^2_{\phi, \psi}(\Omega))^3\Big\}.
\]
The space $V_{\phi,\psi}$ is equipped with the norm
\[
\|u\|_{V_{\phi,\psi}}=
\Big(\|u\|^2_{\phi, \psi}+
\Big\|\sqrt y{\f{\partial u}{\partial x}}\Big\|^2_{\phi, \psi}+
\Big\|\sqrt {y}\, {\f{\partial u}{\partial y}}\Big\|^2_{\phi, \psi}\Big)^
{\frac{1}{2}}.
\]
The sesquilinear form associated to $\mathcal {L}^H$ in $L^2_{\phi ,\psi}(\Omega)$ is given by
\begin{equation}\label{eq3}
a^{\phi, \psi}_H(u,v)=
\int_{\Omega}(\mathcal {L}^H u)(x,y) \overline{v}(x,y)\phi^2(x)\psi^2(y) dx\,dy,
\qquad u,v\in C^{\infty}_c(\Omega).
\end{equation}

We note first the following standard result.

\begin{lemma}\label{l0}
The following assertions hold:
\begin{enumerate}
\item[(a)] The space of test functions $C_c^\infty(\Omega)$ is dense in $L^2_{\phi ,\psi}(\Omega)$,
\item[(b)] the space $V_{\phi,\psi}$ equipped with the norm
$\|\cdot \|_{V_{\phi,\psi}}$ is a Hilbert space.
\end{enumerate}
\end{lemma}
\begin{proof}
Let $u\in L^2_{\phi ,\psi}(\Omega)$. Then $u\phi \psi\in L^2(\Omega)$ and so, for any $\varepsilon >0$ there is $\varphi\in C_c^\infty(\Omega)$ such that $\|\varphi -u\phi \psi\|_{L^2}=\|\phi^{-1}\psi^{-1}\varphi -u\|_{\phi ,\psi}<\varepsilon$. Since $\phi^{-1}\psi^{-1}\varphi\in C_c(\Omega)$, we deduce that $C_c(\Omega)$ is dense in $L^2_{\phi ,\psi}(\Omega)$. Thus the assertion (a) follows by standard mollifier argument.

To prove (b) we have only to show that $V_{\phi,\psi}$ equipped with the norm
$\|\cdot \|_{V_{\phi,\psi}}$ is complete. Consider a Cauchy sequence $(u_n)$ in $(V_{\phi,\psi}, \|\cdot \|_{V_{\phi,\psi}})$. Since $y\ge a$, it follows that $V_{\phi,\psi}$ is continuously embedded in the classical weighted Sobolev space
$$H^1_{\phi ,\psi}(\Omega):=
\Big\{v\,\Big|\quad \Big(v,{\frac{\partial v}{\partial x}},
{\frac{\partial v}{\partial y}}\Big)
\in (L^2_{\phi, \psi}(\Omega))^3\Big\}.$$ Hence, $u_n$ converges to some $u\in H^1_{\phi ,\psi}(\Omega)$. On the other hand, by the convergence of $\sqrt{y}{\frac{\partial u_n}{\partial x}}$ and $\sqrt{y}{\frac{\partial u_n}{\partial y}}$ in $L^2_{\phi, \psi}(\Omega)$ (and hence a.e. by taking a subsequence), it follows that $u\in V_{\phi,\psi}$ and $u_n$ converges to $u$ with respect to the norm $\|\cdot \|_{V_{\phi,\psi}}$.
\end{proof}

The following lemma shows that $a_H^{\phi ,\psi}$ can be extended continuously to a sesquilinear form on
$V_{\phi,\psi}^0\times V_{\phi,\psi}^0$, where $V_{\phi,\psi}^0$ denotes the closure of $C^{\infty}_c(\Omega)$ in $V_{\phi,\psi}$
\begin{lemma}\label{l1}
There is a positive constant $M$ such that
$$|a_H^{\phi ,\psi}(u,v)|\le M \|u\|_{V_{\phi ,\psi}}\|v\|_{V_{\phi ,\psi}},\quad \forall u,\,v\in V_{\phi,\psi}^0.$$
\end{lemma}
\begin{proof}
Integrating by parts, it follows from (\ref{eq3}) that
\begin{eqnarray*}
a_H^{\phi ,\psi}(u,v) &=& \frac{1}{2}\int_\Omega y\frac{\partial u}{\partial x}\frac{\partial \overline{v}}{\partial x}\phi^2\psi^2 +\int_\Omega y\frac{\partial u}{\partial x}\overline{v}\left(\frac{\phi '}{\phi}\right)\phi^2\psi^2+\frac{\sigma^2}{2}\int_\Omega y\frac{\partial u}{\partial y}\frac{\partial \overline{v}}{\partial y}\phi^2\psi^2\\
& & +\frac{\sigma^2}{2}\int_\Omega \frac{\partial u}{\partial y}\overline{v}\phi^2\psi^2+\mu \sigma^2\int_\Omega y^2\frac{\partial u}{\partial y}\overline{v}\phi^2 \psi^2+2\rho \sigma \int_\Omega y\frac{\partial u}{\partial y}\overline{v}\left(\frac{\phi '}{\phi}\right)\phi^2\psi^2\\
& & +\rho \sigma \int_\Omega y\frac{\partial u}{\partial y}\frac{\partial \overline{v}}{\partial x}\phi^2\psi^2-\int_\Omega(\omega \rho \sigma y^2-\frac{1}{2}y+r)\frac{\partial u}{\partial x}\overline{v}\phi^2\psi^2\\
& & -\int_\Omega[\omega \sigma^2 y^2+\kappa(m-y)]\frac{\partial u}{\partial y}\overline{v}\phi^2\psi^2\\
& & -\int_\Omega \left[\frac{1}{2}\omega \sigma^2 y(\omega y^2+1)+\omega y\kappa(m-y)-r\right]u\overline{v}\phi^2\psi^2
\end{eqnarray*}
holds for $u,\,v\in C_c^\infty(\Omega)$.
By H\"older's inequality, and since $\frac{y}{a}\ge 1$ for $y\in [a,\infty),\,a>0$, we have
$$\left|\int_\Omega y\frac{\partial u}{\partial x}\frac{\partial \overline{v}}{\partial x}\phi^2\psi^2 \right|\le \|u\|_{V_{\phi ,\psi}}\|v\|_{V_{\phi ,\psi}},\,\left|\int_\Omega y\frac{\partial u}{\partial y}\frac{\partial \overline{v}}{\partial y}\phi^2\psi^2 \right|\le \|u\|_{V_{\phi ,\psi}}\|v\|_{V_{\phi ,\psi}},$$
$$\left|\int_\Omega y\frac{\partial u}{\partial y}\frac{\partial \overline{v}}{\partial x}\phi^2\psi^2 \right|\le \|u\|_{V_{\phi ,\psi}}\|v\|_{V_{\phi ,\psi}},\,\,\left|\int_\Omega \frac{\partial u}{\partial x}\overline{v}\phi^2\psi^2\right|\le \frac{1}{\sqrt{a}}\|u\|_{V_{\phi ,\psi}}\|v\|_{V_{\phi ,\psi}},\,\hbox{\ and }$$
$$\left|\int_\Omega \frac{\partial u}{\partial y}\overline{v}\phi^2\psi^2\right|\le \frac{1}{\sqrt{a}}\|u\|_{V_{\phi ,\psi}}\|v\|_{V_{\phi ,\psi}}.$$
Since $\psi '(y)=\mu y\psi(y)$, it follows that
\begin{eqnarray}
\int_\Omega yu\overline{v}\phi^2\psi^2 &=& -\frac{1}{2\mu}\left(\int_\Omega \frac{\partial u}{\partial y}\overline{v}\phi^2\psi^2+ \int_\Omega u\frac{\partial \overline{v}}{\partial y}\phi^2\psi^2\right),\label{eq5} \\
\int_\Omega y^2u\overline{v}\phi^2\psi^2 &=& -\frac{1}{2\mu}\left(\int_\Omega y\frac{\partial u}{\partial y}\overline{v}\phi^2\psi^2+ \int_\Omega y\frac{\partial \overline{v}}{\partial y}u\phi^2\psi^2+\int_\Omega u\overline{v}\phi^2\psi^2\right),\label{eq6} \\
\int_\Omega y^3u\overline{v}\phi^2\psi^2 &=& -\frac{1}{2\mu}\left(2\int_\Omega yu\overline{v}\phi^2\psi^2+\int_\Omega y^2\frac{\partial u}{\partial y}\overline{v}\phi^2\psi^2+ \int_\Omega y^2\frac{\partial \overline{v}}{\partial y}u\phi^2\psi^2\right)\label{eq7}.
\end{eqnarray}
Thus it suffices to estimate the integrals
$$\int_\Omega y\frac{\partial u}{\partial y}\overline{v}\phi^2\psi^2,\, \int_\Omega y^2\frac{\partial u}{\partial y}\overline{v}\phi^2\psi^2,\,
\int_\Omega y\frac{\partial u}{\partial x}\overline{v}\phi^2\psi^2,\,\hbox{\ and }\int_\Omega y^2\frac{\partial u}{\partial x}\overline{v}\phi^2\psi^2.$$
Applying (\ref{eq5}) and H\"older's inequality we have
\begin{eqnarray*}
\left|\int_\Omega y\frac{\partial u}{\partial y}\overline{v}\phi^2\psi^2\right| &\le & \|u\|_{V_{\phi ,\psi}}\|\sqrt{y}v\|_{\phi ,\psi}\le \frac{1}{a\mu}\|u\|_{V_{\phi ,\psi}}\|v\|_{V_{\phi ,\psi}},\\
\left|\int_\Omega y\frac{\partial u}{\partial x}\overline{v}\phi^2\psi^2\right| &\le & \|u\|_{V_{\phi ,\psi}}\|\sqrt{y}v\|_{\phi ,\psi}\le \frac{1}{a\mu}\|u\|_{V_{\phi ,\psi}}\|v\|_{V_{\phi ,\psi}}.
\end{eqnarray*}
On the other hand, applying again H\"older's inequality we get
\begin{eqnarray*}
\left|\int_\Omega y^2\frac{\partial u}{\partial y}\overline{v}\phi^2\psi^2\right| &\le & \|u\|_{V_{\phi ,\psi}}\|y^{\frac{3}{2}}v\|_{\phi ,\psi} \,\hbox{\ and }\\
\left|\int_\Omega y^2\frac{\partial u}{\partial x}\overline{v}\phi^2\psi^2\right| &\le & \|u\|_{V_{\phi ,\psi}}\|y^{\frac{3}{2}}v\|_{\phi ,\psi}.
\end{eqnarray*}
It remains to estimate $\|y^{\frac{3}{2}}v\|_{\phi ,\psi}$. It follows from (\ref{eq7}) that
\begin{eqnarray*}
\|y^{\frac{3}{2}}v\|_{\phi ,\psi}^2 &\le & \frac{1}{\mu}\left|\int_\Omega y^2\frac{\partial v}{\partial y}\overline{v}\phi^2\psi^2\right|\\
&\le & \frac{1}{2}\|y^{\frac{3}{2}}v\|_{\phi ,\psi}^2 +\frac{1}{2\mu^2}\|\sqrt{y}\frac{\partial v}{\partial y}\|_{\phi ,\psi}^2.
\end{eqnarray*}
Hence,
$$\|y^{\frac{3}{2}}v\|_{\phi ,\psi}\le \frac{1}{\mu}\|\sqrt{y}\frac{\partial v}{\partial y}\|_{\phi ,\psi}.$$
This ends the proof of the lemma.
\end{proof}

\section{Existence and uniqueness of solutions to the variational equation}
The following proposition deals with the quasi-accretivity of the sesquilinear form $a_H^{\phi ,\psi}$.
\begin{proposition}\label{p1}
Assume that \eqref{feller-condition} is satisfied. Then, under appropriate conditions on $\rho ,\,\nu ,\,\mu$ and $\omega$, there are constants $c_1>0$ and $c_2\in \R$ such that
\begin{equation}\label{Garding}
\Re a_H^{\phi ,\psi}(v,v)\ge c_1\|v\|_{V_{\phi ,\psi}}+c_2\|v\|_{\phi ,\psi}^2,\quad \forall v\in V^0_{\phi ,\psi}.
\end{equation}
\end{proposition}
\begin{proof}
%
%
%
%
%
%
%
%
%
%
%
The real part of the quadratic form $a_H^{\phi,\psi}(v,v)$ is given by
\begin{eqnarray}\label{eq8}
\Re a_H^{\phi,\psi}(v,v)&=&
{\frac{1}{2}}\int_{\Omega}y\Big|{\frac{\partial v}{\partial x}}\Big|^2
\phi^2\psi^2+
{\frac{\sigma^2}{2}}\int_{\Omega}y\Big|{\frac{\partial v}{\partial y}}\Big|^2
\phi^2\psi^2 \nonumber
\\
& & \,+\Re\left(\int_{\Omega}y\,{\frac{\partial v}{\partial x}}\overline{v}\left(\frac{\phi'}{\phi}\right)
\phi^2\psi^2\right)+
\Re \left(\int_{\Omega}\Big({\frac{1}{2}}y-\omega\rho\sigma y^2-r\Big)
{\frac{\partial v}{\partial x}}\overline{v}\phi^2\psi^2\right) \nonumber
\\
& & \,+\Re \left(\int_{\Omega}\Big({\frac{\sigma^2}{2}}-\kappa m)\Big)
{\frac{\partial v}{\partial y}}\overline{v}\phi^2\psi^2\right)
+\kappa \Re \left(\int_{\Omega}y
{\frac{\partial v}{\partial y}}\overline{v}\,\phi^2\psi^2\right) \nonumber
\\
& & \,-\omega\sigma^2
\Re \left(\int_{\Omega}y^2{\frac{\partial v}{\partial y}}\overline{v}\phi^2\psi^2\right)
+\sigma^2 \mu \Re \left(\int_{\Omega}y^2{\frac{\partial v}{\partial y}}\overline{v}
\phi^2\psi^2\right) \nonumber
\\
& & \,+\rho\,\sigma \Re \left(\int_{\Omega}y\,{\frac{\partial v}{\partial x}}
{\frac{\partial \overline{v}}{\partial y}}
\phi^2\psi^2\right)+
2\rho\sigma \Re \left(\int_{\Omega}y{\frac{\partial v}{\partial y}}\overline{v}\left(\frac{\phi'}{\phi}\right)
\phi^2\psi^2\right) \nonumber
\\
& & \,-{\frac{1}{2}}\omega^2\sigma^2\int_{\Omega} y^3 |v|^2\,\phi^2\psi^2-
(\omega \kappa m+{\frac{\omega\sigma^2}{2}})
\int_{\Omega}y\,|v|^2\,\phi^2\psi^2 \nonumber
\\
& & \,+\omega \kappa \int_{\Omega} y^2 |v|^2\,\phi^2\psi^2+
r\int_{\Omega}|v|^2\,\phi^2\psi^2 \nonumber
\\
& & \,= I_1+I_2+I_3+I_4+I_5+I_6+I_7+I_8+I_9+I_{10}+I_{11}+I_{12}+I_{13}+I_{14}.
\end{eqnarray}
By the definition of the $L^2_{\phi,\psi}$ - norm
\begin{equation}\label{eq9}
I_1+I_2={\frac{1}{2}}\Big\|\sqrt{y}{\frac{\partial v}{\partial x}}\Big\|^2_{\phi,\psi}+
{\frac{\sigma^2}{2}}\Big\|\sqrt{y}{\frac{\partial v}{\partial y}}\Big\|^2_{\phi,\psi}.
\end{equation}
To estimate the next integrals we use H\"older's and Young's
inequalities as well as integration by parts taking in mind that
$\Re \left({\frac{\partial v}{\partial x}}\overline{v}\right)={\frac{1}{2}}
{\frac{\partial |v|^2}{\partial x}}$,
$\Re \left({\frac{\partial v}{\partial y}}\overline{v}\right)={\frac{1}{2}}
{\frac{\partial |v|^2}{\partial y}}$, $\phi'=(sign\,x)\nu \phi$ and
$\psi'=\mu y \psi$.

\bigskip\noindent
$\bullet\;$
{\it Estimate of $I_3$ }:
\begin{equation}\label{eq10}
|I_3| \le{\frac{1}{2}}\epsilon_1
\Big\|\sqrt{y}{\frac{\partial v}{\partial x}}\Big\|^2_{\phi,\psi}+
{\frac{\nu^2}{2\epsilon_1}}\|\sqrt{y} v\|^2_{\phi,\psi},
\qquad \epsilon_1>0.
\end{equation}

\bigskip\noindent
$\bullet\;$
{\it Estimate of $I_4$ }:
\begin{equation}\label{eq11}
|I_4|  \le  {\frac{\nu}{2}}\|\sqrt{y} v\|^2_{\phi,\psi} +
\omega\rho\sigma\nu
\int_{\Omega}y^2 |v|^2\phi^2\psi^2
+r\nu\| v\|^2_{\phi,\psi}.
\end{equation}

\bigskip\noindent
$\bullet\;$
{\it Estimate of $I_5$ }:
\begin{equation}\label{eq12}
I_5 =\big(\kappa m -{\frac{\sigma^2}{2}}\big  )\mu
\|\sqrt{y} v\|^2_{\phi,\psi}.
\end{equation}

\bigskip\noindent
$\bullet\;$
{\it Estimate of $I_6$ and $I_{10}$ }:
\begin{equation}\label{eq13}
I_6+I_{10}\ge
-\left({\frac{\kappa}{2}}+\rho\sigma\nu\right)
\|v\|^2_{\phi,\psi}-
\left(\kappa \mu+2\rho\sigma\nu\mu\right)\int_\Omega y^2|v|^2\phi^2\psi^2.
\end{equation}

\bigskip\noindent
$\bullet\;$
{\it Estimate of $I_7$ and $I_{8}$ }:
\begin{equation}\label{eq14}
I_7+I_8=
\sigma^2(\mu-\omega)
\Re \left(\int_{\Omega}y^2{\frac{\partial v}{\partial y}}\overline{v}\phi^2\psi^2\right).
\end{equation}

\bigskip\noindent
$\bullet\;$
{\it Estimate of $I_9$ }:
\begin{equation}\label{eq15}
|I_9|\le
{\frac{1}{2}}\epsilon_2
\Big\|\sqrt{y}{\frac{\partial v}{\partial x}}\Big\|^2_{\phi,\psi}+
{\frac{\rho^2\sigma^2}{2\epsilon_2}}
\Big\|\sqrt{y}{\frac{\partial v}{\partial y}}\Big\|^2_{\phi,\psi},
\qquad \epsilon_2>0.
\end{equation}
On the other hand, it follows from (\ref{eq7}) that
\begin{equation}\label{eq16}
\|\sqrt{y} v\|^2_{\phi,\psi}=-\Re \left(\int_\Omega y^2\frac{\partial v}{\partial y}\overline{v}\phi^2\psi^2\right)
-\mu\|y^{\frac {3}{2}} v\|^2_{\phi,\psi}.
\end{equation}
It follows from (\ref{eq8})-(\ref{eq16}) that
\begin{eqnarray*}
\begin{split}
\Re a_H^{\phi,\psi}(v,v)&\ge
\alpha_1
\Big\|\sqrt{y}{\frac{\partial v}{\partial x}}\Big\|^2_{\phi,\psi}+
\alpha_2
\Big\|\sqrt{y}{\frac{\partial v}{\partial y}}\Big\|^2_{\phi,\psi}+
\alpha_3
\|v\|_{\phi,\psi}^2
\\&
+\alpha_4
\int_{\Omega}y^2 |v|^2\phi^2\psi^2
+\alpha_5
\Re \left(\int_{\Omega}y^2\,{\frac{\partial v}{\partial y}}\overline{v}
\phi^2\psi^2\right)
\\&
+\alpha_6
\|y^{\frac {3}{2}} v\|^2_{\phi,\psi},
\end{split}
\end{eqnarray*}
where

\medskip

\begin{itemize}

\medskip
\item[ ]
$\alpha_1={\frac{1}{2}}\Big(1-\epsilon_1-\epsilon_2\Big)$,

\medskip
\item[ ]
$\alpha_2={\frac{\sigma^2}{2}}\big(1-{\frac{\rho^2}{\epsilon_2}}\big)=:\frac{\sigma^2}{2}\tau$,

\medskip
\item[ ]
$\alpha_3=(-r\nu -\frac{\kappa}{2}-\rho \sigma \nu +r)$,

\medskip
\item[ ]
$\alpha_4=(\omega \kappa -\kappa \mu -\omega\rho\sigma\nu -2\rho\sigma\nu \mu)$,

\medskip
\item[ ]
$\alpha_5=
\omega\Big(\kappa m-{\frac{\sigma^2}{2}}\Big)+\sigma^2\mu +\beta -\left(\kappa m-\frac{\sigma^2}{2}\right)\mu $,

\medskip
\item[ ]
$\alpha_6=
\omega \mu \Big(\kappa m+{\frac{\sigma^2}{2}}\Big)
-\omega^2{\frac{\sigma^2}{2}}
+\mu\Big(\beta -\left(\kappa m-\frac{\sigma^2}{2}\right)\mu\Big)$

\medskip

\item[ ]
$\quad\>\,=
\mu\alpha_5+\omega\mu\sigma^2
-\sigma^2\mu^2-\omega^2{\frac{\sigma^2}{2}}$
\end{itemize}
and
\[\beta=\Big({\frac{\nu^2}{2\epsilon_1}}+{\frac{\nu}{2}}\Big).
\]

\medskip\noindent
In order to ensure that the coefficients
$\alpha_1$, $\alpha_2$ are nonnegative we use the assumption $|\rho |<1$ and we take
$\epsilon_1$ and $\epsilon_2$
such that
\[
\rho^2<\epsilon_2<1-\epsilon_1.
\]
Furthermore we take $\omega >\mu$, and
\begin{equation}\label{eq18}
\nu\le{\frac{\kappa(\omega -\mu)}{\rho\sigma(\omega+2\mu)}},
\end{equation}
when $0<\rho <1$ in order to obtain that $\alpha_4\ge0$ for any $|\rho |<1$.

To prove the lemma, we need first to show that $\Big|\int_{\Omega}y^2\,{\frac{\partial v}{\partial y}}\overline{v}
\phi^2\psi^2\Big|$  can be estimated by
$\Big\|\sqrt y{\frac{\partial v}{\partial y}}\Big\|^2_{\phi,\psi}$.
Indeed, by means of H\"older's and Young's inequalities,
\begin{equation}\label{eq19}
\begin{split}
\Big|\int_{\Omega}y^2\,{\frac{\partial v}{\partial y}}\overline{v}
\phi^2\psi^2\Big|&=
\Big|\int_{\Omega}\sqrt y\,{\frac{\partial v}{\partial y}}y^{{\frac {3}{2}}}\overline{v}
\phi^2\psi^2\Big|
\\&
\le {\frac{\epsilon_3}{2}}
\Big\|\sqrt{y}{\frac{\partial v}{\partial y}}\Big\|^2_{\phi,\psi}+
{\frac{1}{2\epsilon_3}}\|y^{{\frac {3}{2}}} v\|^2_{\phi,\psi}
\end{split}
\end{equation}
with any $\epsilon_3>0$.

On the other hand, using the assumption $\kappa m>\frac{\sigma^2}{2}$ and $\omega >\mu$, we deduce $\alpha_5>0$ and hence
$$\omega>{\frac{\Big(\kappa m-{\frac{3}{2}}\sigma^2\Big){\mu}-\beta}
{\kappa m-{\frac{\sigma^2}{2}}}}.$$ So, by (\ref{eq19}), we obtain
\begin{eqnarray}\label{eq-new}
\Re a_H^{\phi,\psi}(v,v)&\ge &
\alpha_1
\Big\|\sqrt{y}{\frac{\partial v}{\partial x}}\Big\|^2_{\phi,\psi}+
\Big(\alpha_2-\alpha_5{\frac{\epsilon_3}{2}}\Big)
\Big\|\sqrt{y}{\frac{\partial v}{\partial y}}\Big\|^2_{\phi,\psi}\nonumber
\\& &
\quad +\alpha_3
\|v\|^2_{\phi,\psi}+
\Big(\alpha_6-{\frac{\alpha_5}{2\epsilon_3}}\Big)\|y^{{\frac {3}{2}}} v\|^2_{\phi,\psi}.
\end{eqnarray}
Choosing
\begin{equation}\label{eq20'}
\epsilon_3<{\frac{2\alpha_2}{\alpha_5}},
\end{equation}
we deduce that $\alpha_2-\alpha_5{\frac{\epsilon_3}{2}}>0$.\\
The next step is to prove that
\begin{equation}\label{eq21}
\alpha_6-{\frac{\alpha_5}{2\epsilon_3}}\ge 0.
\end{equation}
This is equivalent to show that $\omega$ satisfies the inequality
\begin{equation}\label{eq22}
\begin{split}
{\frac{\sigma^2}{2}}\omega^2-&
\Big[\Big(\kappa m-{\frac{\sigma^2}{2}}\Big)
\Big(\mu-{\frac{1}{2\epsilon_3}}\Big)+\sigma^2\mu\Big]\omega+
\sigma^2\mu^2 -\beta\Big(\mu-{\frac{1}{2\epsilon_3}}\Big)+
\\&
+\Big(\kappa m-{\frac{\sigma^2}{2}}\Big)
\Big(\mu-{\frac{1}{2\epsilon_3}}\Big)\mu
-\sigma^2\mu\Big(\mu-{\frac{1}{2\epsilon_3}}\Big)\le 0.
\end{split}
\end{equation}
So we need to have
$$
\Delta_{\omega}:=\Big(\kappa m-{\frac{\sigma^2}{2}}\Big)^2
\Big(\mu-{\frac{1}{2\epsilon_3}}\Big)^2+
\mu\sigma^4\Big(\mu-{\frac{1}{\epsilon_3}}\Big)+
2\beta\sigma^2\Big(\mu-{\frac{1}{2\epsilon_3}}\Big)\ge 0.
$$
Let us observe that (\ref{eq21}) can be rewritten in the following way
$$
\left(\mu-{\frac{1}{2\epsilon_3}}\right)
\alpha_5+\omega\mu\sigma^2
-\sigma^2\mu^2-\omega^2{\frac{\sigma^2}{2}}\ge 0,
$$
from which we can deduce that
$$
\epsilon_3>{\frac{1}{2\mu}},
$$
since $\omega^2{\frac{\sigma^2}{2}} -\omega\mu\sigma^2
+\sigma^2\mu^2=\frac{\sigma^2}{2}\left((\omega -\mu)^2+\mu^2\right)>0$. Thus,
\begin{equation}\label{eq25}
\Delta_{\omega}\ge0 \Longleftrightarrow
\Big(\kappa m-{\frac{\sigma^2}{2}}\Big)^2\ge
{\frac{2\epsilon_3\mu}{2\epsilon_3\mu-1}}
\left[{\frac{2-2\epsilon_3\mu}{2\epsilon_3\mu-1}}-
{\frac{2\beta}{\mu\sigma^2}}\right]\sigma^4=:
g(2\epsilon_3\mu)\sigma^4,
\end{equation}
where \[
g(t)={\frac{(2+c)t-(1+c)t^2}{(t-1)^2}}
\]
with $c={\frac{2\beta}{\mu\sigma^2}}$. On the other hand, by \eqref{feller-condition}, there exists $\delta >0$ such that $\kappa m>(1+2\sqrt{\delta})\frac{\sigma^2}{2}$. Thus, it follows that
\begin{equation}\label{eq25bis}
\left(\kappa m-\frac{\sigma^2}{2}\right)^2>\delta \sigma^4.
\end{equation}
Hence, (\ref{eq25}) holds if $g(2\epsilon_3\mu)\le \delta$. An easy computation shows that if
\begin{equation}\label{eq27}
2\epsilon_3\mu >\overline{t}:=1+\frac{1}{\sqrt{1+\delta}}
\end{equation}
then $g(2\epsilon_3\mu)< \delta$ and therefore $\Delta_\omega > 0$. On the other hand, it follows from (\ref{eq20'}) and (\ref{eq27}) that $\alpha_5<\frac{4\mu\alpha_2}{\overline{t}}$ and therefore, using (\ref{eq18}),
\begin{equation}\label{eq28}
\mu < \omega<{\frac{\Big(\kappa m-{\frac{3}{2}}\sigma^2\Big)\mu+
{\frac{4}{\overline t}}\mu\alpha_2-\beta}{\kappa m-{\frac{\sigma^2}{2}}}}=
\mu+{\frac{\gamma\sigma^2\mu-\beta}{\kappa m-{\frac{\sigma^2}{2}}}},
\end{equation}
where
$\gamma={\frac{2\tau}{\overline t}}-1$.
This implies in particular that $\gamma >0$ and
\begin{equation}\label{eq30}
\mu >\frac{\beta}{\gamma \sigma^2}.
\end{equation}
Thus, using conditions (\ref{eq18}) and (\ref{eq28}), we deduce that (\ref{eq22}) holds if $\omega \in (M,N)$, where
\[
M=\hbox{max}\left\{
{\frac{
\Big(\kappa m-{\frac{\sigma^2}{2}}\Big)
\Big(\mu-{\frac{1}{2\epsilon_3}}\Big)+\sigma^2\mu\
-{\sqrt\Delta_\omega}}
{\sigma^2}},
\mu\right\}
\]
and
\[
N=\hbox{min}\left\{
{\frac{
\Big(\kappa m-{\frac{\sigma^2}{2}}\Big)
\Big(\mu-{\frac{1}{2\epsilon_3}}\Big)+\sigma^2\mu\
+{\sqrt\Delta_\omega}}
{\sigma^2}},
\mu+{\frac{\gamma\sigma^2\mu-\beta}{\kappa m-{\frac{\sigma^2}{2}}}}
\right\}.
\]
Let us observe that
\[
\Big(\kappa m-{\frac{\sigma^2}{2}}\Big)
\Big(\mu-{\frac{1}{2\epsilon_3}}\Big)+\sigma^2\mu\
>{\sqrt\Delta_\omega}
\]
if and only if
\begin{equation}\label{eq30'}
\beta<\mu \Big(\kappa m-{\frac{\sigma^2}{2}}\Big)+
{\frac{\mu\sigma^2}{2\epsilon_3\mu-1}}.
\end{equation}
Moreover, it is easy to see that $\mu \le N$.\\
To get
\[
{\frac{
\Big(\kappa m-{\frac{\sigma^2}{2}}\Big)
\Big(\mu-{\frac{1}{2\epsilon_3}}\Big)+\sigma^2\mu\
-{\sqrt\Delta_\omega}}
{\sigma^2}}
<
\mu+{\frac{\gamma\sigma^2\mu-\beta}{\kappa m-{\frac{\sigma^2}{2}}}} 
\]
or, equivalently,
\begin{equation}\label{eq31}
\Big(\kappa m-{\frac{\sigma^2}{2}}\Big)
\Big(\mu-{\frac{1}{2\epsilon_3}}\Big)-
{\frac{\sigma^2}{\kappa m-{\frac{\sigma^2}{2}}}}
\left(\gamma\sigma^2 \mu-\beta\right)
<{\sqrt\Delta_\omega},
\end{equation}
we firstly require that
\begin{equation}\label{eq32}
\Big(\kappa m-{\frac{\sigma^2}{2}}\Big)^2\ge
{\frac{2\epsilon_3\mu}{2\epsilon_3\mu-1}}
\left(\gamma-
{\frac{\beta}{\mu\sigma^2}}\right)\sigma^4=:
f(2\epsilon_3\mu)\sigma^4
\end{equation}
to have that the left side in (\ref{eq31}) is nonnegative. \\
It follows from (\ref{eq27}) and (\ref{eq30}) that $0<f(2\epsilon_3\mu)$. Thus, from (\ref{eq25bis}) we obtain (\ref{eq32}) if
$f(2\epsilon_3\mu)\le \delta$.
From the definition of $\overline{t}$ and since $\tau <1$ one can see that $\delta >\gamma$.
Using again $\tau <1$, we obtain
$2\tau -1<1<\sqrt{1+\delta}=\frac{\delta}{\sqrt{1+\delta}}+\frac{1}{\sqrt{1+\delta}}$ and so,
\begin{eqnarray*}
\gamma -\frac{c}{2} &< & \gamma =\frac{2\tau}{\overline{t}}-1\\
&<& \frac{\delta}{\overline{t}\sqrt{1+\delta}}\\
&=& \frac{\delta}{1+\sqrt{1+\delta}}.
\end{eqnarray*}
Hence,
$$\sqrt{1+\delta}<\frac{\delta}{\gamma -(c/2)}-1.$$ This implies that 
 $\overline{t}> \frac{\delta}{\delta -\gamma +(c/2)}$. This together with (\ref{eq27}) imply that
$f(2\epsilon_3\mu)\le \delta .$ Thus, (\ref{eq32}) holds.\\
Using now the definition of $\Delta_\omega$, one can see that proving (\ref{eq31}) is equivalent to show
\begin{equation}\label{eq33}
\Big(\kappa m-{\frac{\sigma^2}{2}}\Big)^2>
{\frac{2\epsilon_3\mu}{2\epsilon_3\mu(1+2\gamma)-
(2+2\gamma)}}
\left(\gamma-
{\frac{\beta}{\mu\sigma^2}}\right)^2\sigma^4=
{\tilde f}(2\epsilon_3\mu)\sigma^4.
\end{equation}
Since $\overline{t}<2$ and
$$\overline{t}<\inf_{\gamma \in (0,\frac{2}{\overline{t}}-1)}\left(1+\frac{1}{1+2\gamma}\right)=\frac{4}{4-\overline{t}},$$
one deduces that ${\tilde f}(2\epsilon_3\mu)<0$ and hence (\ref{eq33}) holds, provided that
$$
\overline{t}<2\epsilon_3\mu < 1+\frac{1}{1+2\gamma}.
$$
Therefore, if $\omega \in (M,N),\, \nu$ satisfies (\ref{eq18}) when $\rho >0$, and
$$\beta <\min\left\{\mu \gamma \sigma^2,\mu\left(\kappa m-\frac{\sigma^2}{2}\right)+\frac{\mu \sigma^2}{2\epsilon_3\mu -1}\right\}$$
from (\ref{eq30}) and (\ref{eq30'}), with $0<\gamma <\delta$. Then \eqref{eq-new} can be written as
$$\Re a_H^{\phi ,\psi}(v,v)\ge c_1\|v\|_{V_{\phi ,\psi}}+\alpha_3\|v\|^2_{\phi ,\psi},\quad \forall v\in C_c^\infty(\Omega),$$
provided that
$$ \frac{2\rho^2}{2-\overline{t}}<\epsilon_2<1-\epsilon_1  \hbox{\ and }
\epsilon_3\in\left({\frac{\overline t}{2\mu}},
\hbox{min}\left\{
{\frac{2\alpha_2}{\alpha_5}},
{\frac{1}{2\mu}}
\left(1+{\frac{1}{1+2\gamma}}\right)\right\}\right),
$$
where $c_1:=\min\{\alpha_1,\alpha_2-\alpha_5\frac{\epsilon_3}{2},a^3(\alpha_6-\frac{\alpha_5}{2\epsilon_3})\}>0$.
We note that the above first inequality satisfied by $\epsilon_2$ is a consequence of $\gamma >0$. On the other hand, by assuming $|\rho|<\sqrt{\frac{1}{2}-\frac{1}{2\sqrt{1+\delta}}}$, there exists a $\epsilon_1$ satisfying the above condition, since
$\sqrt{\frac{1}{2}-\frac{1}{2\sqrt{1+\delta}}}=\sqrt{\frac{2-\overline{t}}{2}}.$
\end{proof}
\begin{remark}\label{r1}
It follows from Lemma \ref{l1} and Proposition \ref{p1} that the form norm defined by $$\|u\|_{a_H}:=\sqrt{\Re a_H^{\phi ,\psi}(u,u)+(1-c_2)\|u\|_{\phi ,\psi}},$$
is equivalent to the norm $\|\cdot \|_{V_{\phi ,\psi}}$. So, by Lemma \ref{l0}, the sesquilinear form $a_H^{\phi ,\psi}$ with domain $V^0_{\phi ,\psi}$ is closed.
\end{remark}

We define the operator associated to $a_H^{\phi ,\psi}$ by
\begin{eqnarray*}
D(A) &=& \{u\in V^0_{\phi ,\psi}\,s.t.\,\exists v\in L^2_{\phi ,\psi}(\Omega): a_H^{\phi ,\psi}(u,\varphi)=\int_\Omega v\overline{\varphi}\phi^2\psi^2,\,\forall \varphi \in C_c^\infty(\Omega)\} \\
Au &=& v.
\end{eqnarray*}

The estimate \eqref{Garding} is known as Garding's inequality. Applying \cite[Section 4.4, Theorem 4.1]{LM} we obtain the existence of a unique weak solution to the problem \eqref{eq2}.
\begin{theorem}
Assume the same conditions as in Proposition \ref{p1}. Then, there is a unique weak solution $u\in L^2([0,T],V^0_{\phi ,\psi})\cap C([0,T],L^2_{\phi ,\psi}(\Omega))$ to the parabolic problem \eqref{eq2}.
\end{theorem}

Applying the Lumer-Phillips theorem we obtain the following generation result.
\begin{theorem}
Assume the same conditions as in Proposition \ref{p1}. Then, the operator $-A$ defined above generates a positivity preserving and quasi-contractive analytic semigroup on $L^2_{\phi ,\psi}(\Omega)$.
\end{theorem}
\begin{proof}
It follows form Lemma \ref{l0}, Lemma \ref{l1}, Proposition \ref{p1} and Remark \ref{r1} that the form $a_H^{\phi ,\psi}$ with domain $V^0_{\phi ,\psi}$ is densely defined, closed, continuous and quasi-accretive sesquilinear form on $L^2_{\phi ,\psi}(\Omega)$. Thus, $-A$ generates a quasi-contractive analytic semigroup $(e^{-tA})_{t\ge 0}$ on $L^2_{\phi ,\psi}(\Omega)$ (cf. \cite[Theorem 1.52]{Ouhabaz}).

For the positivity, we note first that the semigroup $(e^{-tA})_{t\ge 0}$ is real and one can see that for every
$u\in D(a_H^{\phi ,\psi})\cap L^2_{\phi ,\psi}(\Omega ,\R),\,u^+ \in D(a_H^{\phi ,\psi})$ and $a_H^{\phi ,\psi}(u^+,u^-)=0$,
since $u^-=(-u)^+$ and $\nabla u^+=\chi_{\{u>0\}}\nabla u$ (cf. \cite[Proposition 4.4]{Ouhabaz}). Thus, by the first Beurling-Deny criteria, $(e^{-tA})_{t\ge 0}$ is a positivity preserving semigroup on $L^2_{\phi ,\psi}(\Omega)$ (cf. \cite[Theorem 2.6]{Ouhabaz}).

\end{proof}




\end{document}